\documentclass[10pt,reqno]{amsart}
\usepackage[active]{srcltx} 
\usepackage{latexsym}
\usepackage{amssymb,amsmath,amsfonts,bm}
\usepackage[ansinew]{inputenc}
\usepackage{graphicx}
\usepackage{color}
\usepackage{url}
\usepackage{enumerate}
\usepackage[colorlinks]{hyperref}
\usepackage{pdftricks}
\numberwithin{equation}{section}
\theoremstyle{plain}
\newtheorem{thm}{Theorem}[section]

\newtheorem{lem}[thm]{Lemma}

\theoremstyle{definition}

\theoremstyle{remark}

\usepackage{fancyhdr}
\usepackage{graphicx}
\begin{document}

\title{Infinite and finite dimensional Hilbert tensors}%
 \thanks{Email: songyisheng1@gmail.com(Song); maqilq@polyu.edu.hk(Qi)}
 \thanks{The work was supported by the Hong Kong Research Grant Council (Grant No. PolyU 502510, 502111, 501212, 501913),  the first author was supported partly by the National Natural Science Foundation of P.R. China (Grant No. 11171094, 11271112) and by the Research Projects of Department of Science and Technology of Henan Province (Grant No. 122300410414,132300410432).}
 \maketitle
 \begin{center}{Yisheng Song$^1$ and Liqun Qi$^2$}\\\vskip 2mm
 1. School of Mathematics and Information Science,
 Henan Normal University, XinXiang, HeNan,  P.R. China, 453007.\\\vskip 2mm
  2. Department of Applied Mathematics, The Hong Kong Polytechnic University, Hung Hom, Kowloon, Hong Kong\end{center}
 %
\vskip 4mm
\begin{quote}
{\bf Abstract.}\ For an $m$-order $n-$dimensional Hilbert  tensor
(hypermatrix) $\mathcal{H}_n=(\mathcal{H}_{i_1i_2\cdots i_m})$,
$$\mathcal{H}_{i_1i_2\cdots i_m}=\frac1{i_1+i_2+\cdots+i_m-m+1},\
i_1,\cdots, i_m=1,2,\cdots,n$$ its spectral radius is not larger
than $n^{m-1}\sin\frac{\pi}{n}$, and an upper bound of its
$E$-spectral radius is $n^{\frac{m}2}\sin\frac{\pi}{n}$. Moreover,
its spectral radius is strictly increasing and  its $E$-spectral
radius is nondecreasing with respect to the dimension $n$.  When the
order is even, both infinite and finite dimensional Hilbert tensors
are positive definite.   We also show that the $m$-order infinite
dimensional Hilbert  tensor (hypermatrix)
$\mathcal{H}_\infty=(\mathcal{H}_{i_1i_2\cdots i_m})$ defines a
bounded and positively $(m-1)$-homogeneous operator from $l^1$ into
$l^p$ ($1<p<\infty$), and the norm of corresponding positively
homogeneous operator is smaller than or equal to
$\frac{\pi}{\sqrt6}$.
\end{quote}

\pagestyle{fancy} \fancyhead{} \fancyhead[EC]{ Yisheng Song and
Liqun Qi} \fancyhead[EL,OR]{\thepage} \fancyhead[OC]{Hilbert
tensors} \fancyfoot{}

\section{\bf Introduction}\label{}
In linear algebra, an $n$-dimensional Hilbert matrix $H_n=(H_{ij})$
is a square matrix with entries being the unit fractions, i.e.,
$$H_{ij}=\frac1{i+j-1},\ i,j=1,2,\cdots,n,$$
which was introduced by Hilbert \cite{H1894}. Clearly, an
$n$-dimensional Hilbert matrix is symmetric and positive definite,
and is a compact linear operator on finite dimensional space. Many
nice properties of $n$-dimensional Hilbert matrix have been
inivestigated by Frazer\cite{F1946} and Taussky \cite{T1949}. An
infinite dimensional Hilbert matrix
$$H_\infty=\left(\frac1{i+j-1}\right),\ i,j=1,2,\cdots,n,\cdots$$
can be regarded as a bounded linear operator from Hilbert space
$l^2$ into itself (here, $l^p$ ($0<p<\infty$) is a space consisting
of all sequences $x=(x_i)_{i=1}^\infty$ satisfying
$\sum\limits_{i=1}^{\infty}|x_i|^p<+\infty$), but not compact
operator (Choi \cite{C1983}) and Ingham \cite{I1936}). The spectral
properties of infinite dimensional Hilbert matrix have been studied
by  Magnus \cite{M1950} and Kato \cite{K1957}.

As a natural extension of a Hilbert matrix, the entries of an
$m$-order $n$-dimensional Hilbert tensor (hypermatrix)
$\mathcal{H}_n=(\mathcal{H}_{i_1i_2\cdots i_m})$ are defined by
$$\mathcal{H}_{i_1i_2\cdots i_m}=\frac1{i_1+i_2+\cdots+i_m-m+1},\ i_1,i_2,\cdots, i_m=1,2,\cdots,n.$$
 The entries of an $m$-order infinite dimensional Hilbert  tensor (hypermatrix) $\mathcal{H}_\infty=(\mathcal{H}_{i_1i_2\cdots i_m})$ are defined by
$$\mathcal{H}_{i_1i_2\cdots i_m}=\frac1{i_1+i_2+\cdots+i_m-m+1},\ i_1,i_2,\cdots, i_m=1,2,\cdots,n,\cdots.$$
The Hilbert tensor may be regarded as derived from the integral
\begin{equation}\label{e00} \mathcal{H}_{i_1i_2\cdots i_m}=\int_0^1 t^{i_1+i_2+\cdots+i_m-m}dt. \end{equation}

Clearly, both $\mathcal{H}_n$ and  $\mathcal{H}_\infty$ are positive ($\mathcal{H}_{i_1i_2\cdots i_m}>0$) and symmetric ($\mathcal{H}_{i_1i_2\cdots i_m}$ are invariant for any permutation of the indices), and an $m$-order $n$-dimensional  Hilbert tensor $\mathcal{H}_n$ is a Hankel tensor with $v=(1,\frac12, \frac13, \cdots,\frac{1}{nm})$ (Qi \cite{LQH}), and an $m$-order  infinite dimensional  Hilbert tensor $\mathcal{H}_\infty$ is a Hankel tensor with $v=(1,\frac12, \frac13, \cdots,\frac{1}{n},\cdots)$.\\

For a vector $x = (x_1, x_2,\cdots, x_n)^T\in \mathbb{R}^n$, $\mathcal{H}_nx^{m-1}$ is a vector defined by
\begin{equation}\label{eq:11}(\mathcal{H}_nx^{m-1})_i=\sum_{i_2,\cdots,i_m=1}^n\frac{x_{i_2}\cdots x_{i_m}}{i+i_2+\cdots+i_m-m+1},\  i=1,2,\cdots,n.\end{equation}
Then $x^T(\mathcal{H}_nx^{m-1})$ is a homogeneous polynomial,  denoted $\mathcal{H}_nx^m$, i.e., \begin{equation}\label{eq:12}\mathcal{H}_nx^m=x^T(\mathcal{H}_nx^{m-1})=\sum_{i_1,i_2,\cdots,i_m=1}^n\frac{x_{i_1}x_{i_2}\cdots x_{i_m}}{i_1+i_2+\cdots+i_m-m+1},\end{equation}
where $x^T$ is the transposition of $x$.

For a real vector $x = (x_1, x_2,\cdots, x_n,x_{n+1},\cdots)\in l^1$
(here $l^1$ is a space of sequences whose series is absolutely
convergent), $\mathcal{H}_\infty x^{m-1}$ is an infinite dimensional
vector defined by
\begin{equation}\label{eq:13}(\mathcal{H}_\infty x^{m-1})_i=\sum_{i_2,\cdots,i_m=1}^\infty\frac{x_{i_2}\cdots x_{i_m}}{i+i_2+\cdots+i_m-m+1},\ i=1,2,\cdots.\end{equation}
Accordingly,  $\mathcal{H}_\infty x^m$ is given by
\begin{equation}\label{eq:14}\mathcal{H}_\infty
x^m=\lim_{n\to\infty}\mathcal{H}_n
x^m=\sum_{i_1,i_2,\cdots,i_m=1}^\infty\frac{x_{i_1}x_{i_2}\cdots
x_{i_m}}{i_1+i_2+\cdots+i_m-m+1}.\end{equation} Then
$\mathcal{H}_\infty x^m$ is exactly a real number for each  real
vector $x\in l^1$, i.e., $\mathcal{H}_\infty x^m<\infty$. In fact,
since $\sum\limits_{i=1}^\infty|x_i|<\infty$ for $x = (x_1,
x_2,\cdots, x_n,x_{n+1},\cdots)\in l^1$, we have
$$\begin{aligned}\mathcal{H}_\infty x^m=\lim_{n\to\infty}\mathcal{H}_nx^m=&\lim_{n\to\infty}\sum_{i_1,i_2,\cdots,i_m=1}^n\frac{x_{i_1}x_{i_2}\cdots x_{i_m}}{i_1+i_2+\cdots+i_m-m+1}\\
\leq&\lim_{n\to\infty} \sum_{i_1,i_2,\cdots,i_m=1}^n\frac{|x_{i_1}x_{i_2}\cdots x_{i_m}|}{\underbrace{1+1+\cdots+1}_{m}-m+1}\\
=& \lim_{n\to\infty}\sum_{i_1,i_2,\cdots,i_m=1}^n|x_{i_1}||x_{i_2}|\cdots |x_{i_m}|\\
=&\lim_{n\to\infty}\left(\sum_{i=1}^n|x_i|\right)^m=\left(\sum_{i=1}^\infty|x_i|\right)^m<\infty.
\end{aligned}$$
In Section 2, we will prove that $\mathcal{H}_\infty x^{m-1}$ is
well defined, i.e., $\mathcal{H}_\infty x^{m-1}\in l^p$
$(1<p<\infty)$ for all real vector $x\in l^1$.

Both infinite and finite dimensional Hilbert tensors $\mathcal{H}_n$
and $\mathcal{H}_\infty$ are positive tensors.   Thus, they are
strictly copositive, i.e., $$\mathcal{H}_n x^m> 0 \mbox{ for all
}x\in \mathbb{R}_+^n\setminus\{\theta\}$$ and $$\mathcal{H}_\infty
x^m> 0\mbox{ for all real nonnegative vector }x\in
l^1\setminus\{\theta\},$$ where $\theta$ is zero vector with all
entries being $0$ and $\mathbb{R}^n_{+}=\{x\in
\mathbb{R}^n;x_i\geq0,\ i=1,2,\cdots,n\}$.  The concept of
(strictly) copositive tensors was introduced and used by Qi
\cite{LQ13}.

When the order $m$ is even, both infinite and finite dimensional
Hilbert tensors are positive definite.  The concept of positive
(semi-)definite tensors was introduced by Qi \cite{LQ1}.
\begin{thm} Let $m,n$ be two positive integers and $m$ be even. Then both $m-$order Hilbert tensors $\mathcal{H}_n$ and $\mathcal{H}_\infty$ are positive definite, i.e.,
$$\mathcal{H}_n x^m> 0 \mbox{ for all }x\in \mathbb{R}^n\setminus\{\theta\}$$ and $$\mathcal{H}_\infty x^m>  0\mbox{ for all real vector }x\in l^1\setminus\{\theta\}.$$
\end{thm}
\begin{proof}
By (\ref{e00}), for each positive integer $n$ and $x\in
\mathbb{R}^n$, we have
$$\begin{aligned}\mathcal{H}_nx^m=&\sum_{i_1,i_2,\cdots,i_m=1}^n\int_0^1t^{i_1+i_2+\cdots +i_m-m}x_{i_1}x_{i_2}\cdots x_{i_m}dt\\
= & \int_0^1\sum_{i_1,i_2,\cdots,i_m=1}^n (\Pi_{j=1}^mt^{i_j-1}x_{i_j})dt\\
=& \int_0^1\left(\sum_{i=1}^nt^{i-1}x_i\right)^mdt\\
\ge&0.
\end{aligned}$$
This shows that $\mathcal{H}_n$ is positive semi-definite.

Now we assume that $\mathcal{H}_n$ is not positive definite.  Then
there exists $\bar x\in \mathbb{R}^n\setminus\{\theta\}$ such that
$\mathcal{H}_n x^m = 0$.  Then from the derivation in the last
paragraph, we see that
\[
\int_0^1 \left(\sum_{i=1}^n t^{i-1}\bar x_i\right)^m dt=0.
\]
By the continuity, we have
\[
\sum_{i=1}^n t^{i-1}\bar x_i \equiv 0 \mbox{ for all } t \in [0,1].
\]
Letting $t=0$, we have $\bar x_1=0$ and so,
\[
t (\bar x_2 +t \bar x_1 +\ldots + t^{n-2} \bar x_n)=0  \mbox{ for
all } t \in [0,1].
\]
So, for all  $t \in (0,1]$, we have
\[
\bar x_2 +t \bar x_1 +\ldots + t^{n-2} \bar x_n=0.
\]
Again by continuity, we see that
\[
\bar x_2 +t \bar x_1 +\ldots + t^{n-2} \bar x_n=0 \mbox{ for all } t
\in [0,1].
\]
Letting $t=0$, we see that $\bar x_2=0$. Repeating this process, we
see that
\[
\bar x_i=0 \mbox{ for all } i=1,\ldots,n.
\]
Therefore, $\bar x=\theta$, which forms a contradiction.   Hence
$\mathcal{H}_n$ is positive definite.

Similarly, we can show that $\mathcal{H}_\infty$ is positive
definite.
\end{proof}

In the remainder of this paper, we will investigate some other nice properties of infinite and finite dimensional Hilbert tensors such as spectral radius and operator norm and so on. \\

 In Section 2, we will prove that the $m$-order infinite dimensional Hilbert  tensor (hypermatrix) $\mathcal{H}_\infty=(\mathcal{H}_{i_1i_2\cdots i_m})$
defines a bounded and positively $(m-1)$-homogeneous operator from $l^1$ into  $l^p$ ($1<p<\infty$).
When $(\mathcal{H}_\infty x^{m-1})^{[\frac1{m-1}]}$ is well defined for all real vector $x\in l^1$,  let \begin{equation}F_\infty x=(\mathcal{H}_\infty x^{m-1})^{[\frac1{m-1}]}\mbox{ and } T_\infty x= \begin{cases}\|x\|_1^{2-m}\mathcal{H}_\infty x^{m-1},\ x\neq\theta\\
\theta,\ \ \ \ \ \ \ \ \ \ \  \ \ \ \ \ \ \ \ \  \ \  x=\theta,\end{cases}\end{equation} where $x^{[\frac1{m-1}]}=(x_1^{\frac1{m-1}},x_2^{\frac1{m-1}},\cdots , x_n^{\frac1{m-1}},\cdots)$ and $\theta$ is zero vector with entries being all $0$. We will show that  $T_\infty$ is a bounded, continuous and positively homogeneous operator from $l^1$ into  $l^p$ ($1<p<\infty$) and  $F_\infty$ is a bounded, continuous and positively homogeneous operator from $l^1$ into  $l^p$ ($m-1<p<\infty$).  Furthermore, their norms are not larger than $\frac{\pi}{\sqrt6}$. \\

In Section 3, we will study the spectral properties of an $m$-order
$n-$dimensional Hilbert tensor  $\mathcal{H}_n$. With the help of
the finite dimensional Hilbert inequality, the largest
$H-$eigenvalue (spectral radius) of Hilbert tensor $\mathcal{H}_n$
is not exactly larger than $n^{m-1}\sin\frac{\pi}{n}$, and the
largest $Z-$eigenvalue ($E-$spectral radius) of $\mathcal{H}_n$  is
smaller than or equal to  $n^{\frac{m}2}\sin\frac{\pi}{n}$.
Furthermore,  the spectral radius of Hilbert tensor $\mathcal{H}_n$
is strictly increasing with respect to the dimensionality $n$ and
its $E$-spectral radius is nondecreasing with respect to the
dimensionality $n$.


\section{\bf Infinite dimensional Hilbert tensors}

 For $0 < p < \infty$, $l^p$ is the space consisting of all sequences $x = (x_i)$ satisfying
 $$\sum_{i=1}^\infty |x_i|^p < \infty.$$
 If $p \geq 1$, then a norm on $l^p$ is  defined by
 $$\|x\|_p = \left(\sum_{i=1}^\infty |x_i|^p\right)^{\frac1p}.$$
 In fact, the space $(l^p,\|\cdot\|_p)$ is a Banach space for $p \geq 1$.

Let $(X,\|\cdot\|_X)$ and $(Y,\|\cdot\|_Y)$ be two Banach space, and $T:X\to Y$ be an operator and $r$ is a real number. $T$ is called \begin{itemize}
\item  {\em $r$-homogeneous} if $T(tx)=t^rTx$ for each $t\in\mathbb{K}$ and all $x\in X$;
\item  {\em positively homogeneous} if $T(tx)=tTx$ for each $t>0$ and all $x\in X$;
\item  {\em bounded} if there is a real number $M>0$ such that  $$\|Tx\|_Y\leq M\|x\|_X, \mbox{ for all }x\in X.$$
\end{itemize}

 Let $T$ be a bounded, continuous and positively homogeneous operator from  $X$ into $Y$. Then the norm of $T$ can be defined by
\begin{equation}\label{eq:21} \|T\|= \sup\{\|Tx\|_Y: \|x\|_X= 1\}.\end{equation}

When $(\mathcal{H}_\infty x^{m-1})^{[\frac1{m-1}]}$ is well defined for all real vector $x\in l^1$,  let \begin{equation}\label{eq:22}F_\infty x=\left(\mathcal{H}_\infty x^{m-1}\right)^{[\frac1{m-1}]}\end{equation}  and \begin{equation}\label{eq:23}T_\infty x= \begin{cases}\left\|x\right\|_1^{2-m}\mathcal{H}_\infty x^{m-1},\ x\neq\theta\\
\theta,\ \ \ \ \ \ \ \ \ \ \  \ \ \ \ \ \ \ \ \  \ \  x=\theta,\end{cases}\end{equation} where $x^{[\frac1{m-1}]}=(x_1^{\frac1{m-1}},x_2^{\frac1{m-1}},\cdots , x_n^{\frac1{m-1}},\cdots)$ and $\theta$ is zero vector with entries being all $0$. Clearly, both $F_\infty$ and $T_\infty$ are continuous and positively homogeneous. With the help of the well known series
$$\sum\limits_{i=1}^{\infty}\frac1{i^p}<\infty\mbox{ for $\infty>p>1$ and }\sum\limits_{i=1}^{\infty}\frac1{i^2}=\frac{\pi^2}{6},$$ now we discuss properties of  the infinite dimensional Hilbert tensor.
\begin{thm}\label{th:21} Let $F_\infty$ and $T_\infty$ be  defined by the equations (\ref{eq:22}) and (\ref{eq:23}), respectively. Then \begin{itemize}
\item[(i)] if $x\in l^1$, then $T_\infty x\in l^p$ for $1<p<\infty$;
\item[(ii)] if $x\in l^1$, then $F_\infty x\in l^p$ for $m-1<p<\infty$.
\end{itemize}
Furthermore, $T_\infty$ is a bounded, continuous and positively homogeneous operator from $l^1$ into  $l^p$ ($1<p<\infty$) and  $F_\infty$ is a bounded, continuous and positively homogeneous operator from $l^1$ into  $l^p$ ($m-1<p<\infty$). In particular,
 $$\|T_\infty\|=\sup_{\|x\|_1= 1}\|T_\infty x \|_2\leq \frac{\pi}{\sqrt6}$$ and
 $$\|F_\infty\|=\sup_{\|x\|_1= 1}\|F_\infty x \|_{2(m-1)}\leq \frac{\pi}{\sqrt6}.$$
 \end{thm}
\begin{proof} For $x\in l^1$,
$$\begin{aligned}|(\mathcal{H}_\infty x^{m-1})_i|=&\lim_{n\to\infty}\left|\sum_{i_2,\cdots,i_m=1}^n\frac{x_{i_2}\cdots x_{i_m}}{i+i_2+\cdots+i_m-m+1}\right|\\
\leq&\lim_{n\to\infty} \sum_{i_2,\cdots,i_m=1}^n\frac{|x_{i_2}\cdots x_{i_m}|}{i+\underbrace{1+\cdots+1}_{m-1}-m+1}\\
=&\frac1i\lim_{n\to\infty}\sum_{i_2,\cdots,i_m=1}^n|x_{i_2}||x_{i_3}|\cdots |x_{i_m}|\\
=&\frac1i\lim_{n\to\infty}\left(\sum_{k=1}^n|x_k|\right)^{m-1}\\
=&\frac1i\left(\sum_{k=1}^\infty|x_k|\right)^{m-1}\\
=&\frac1i\|x\|_1^{m-1}.
\end{aligned}$$
Then (i) for $p>1$, it follows from the definition of  $T_\infty$ that
$$\begin{aligned}\sum_{i=1}^{\infty}|(T_\infty x)_i|^p=&\sum_{i=1}^{\infty}|(\|x\|_1^{2-m}\mathcal{H}_\infty x^{m-1})_i|^p\\
=&\|x\|_1^{(2-m)p}\sum_{i=1}^{\infty}|(\mathcal{H}_\infty x^{m-1})_i|^p\\
\leq&\|x\|_1^{(2-m)p}\sum_{i=1}^{\infty}(\frac1i\|x\|_1^{m-1})^p\\
=&\|x\|_1^p\sum_{i=1}^{\infty}\frac1{i^p}<\infty
\end{aligned}$$
since  $\sum_{i=1}^{\infty}\frac1{i^p}<\infty$ for $p>1$, and thus  $T_\infty x\in l^p$ for all $x\in l^1$. Moreover, we also have
\begin{equation}\label{eq:24}
\|T_\infty x\|_p=\left(\sum_{i=1}^{\infty}|(T_\infty x)_i|^p\right)^\frac1p\leq M\|x\|_1,
\end{equation} where $M=\left(\sum\limits_{i=1}^{\infty}\frac1{i^p}\right)^\frac1p>0$. So, $T_\infty$ is a bounded operator from $l^1$ into $l^p$ ($1<p<\infty$).  In particular, take $p=2$, $M=\left(\sum\limits_{i=1}^{\infty}\frac1{i^2}\right)^\frac12=\frac{\pi}{\sqrt6}.$ It follows from (\ref{eq:21}) and (\ref{eq:24}) that
$$\|T_\infty\|=\sup_{\|x\|_1= 1}\|T_\infty x \|_2\leq\frac{\pi}{\sqrt6}.$$
(ii) for $p>m-1$, it follows from the definition of  $F_\infty$ that
$$\begin{aligned}\sum_{i=1}^{\infty}|(F_\infty x)_i|^p=&\sum_{i=1}^{\infty}|(\mathcal{H}_\infty x^{m-1})_i|^{\frac{p}{m-1}}\\
\leq&\sum_{i=1}^{\infty}(\frac1i\|x\|_1^{m-1})^{\frac{p}{m-1}}\\
=&\|x\|_1^p\sum_{i=1}^{\infty}\frac1{i^{\frac{p}{m-1}}}<\infty
\end{aligned}$$
since  $\sum_{i=1}^{\infty}\frac1{i^{\frac{p}{m-1}}}<\infty$ for $p>m-1$, and hence  $F_\infty x\in l^p$ for all $x\in l^1$. Moreover, we also have
\begin{equation}\label{eq:25}
\|T_\infty x\|_p=\left(\sum_{i=1}^{\infty}|(F_\infty x)_i|^p\right)^\frac1p\leq C\|x\|_1,
\end{equation} where $C=\left(\sum\limits_{i=1}^{\infty}\frac1{i^{\frac{p}{m-1}}}\right)^{\frac{m-1}{p}}>0$. So, $F_\infty$ is a bounded operator from $l^1$ into $l^p$ ($m-1<p<\infty$).  Similarly, take $p=2(m-1)$, $C=\frac{\pi}{\sqrt6}.$ It follows from (\ref{eq:21}) and (\ref{eq:25}) that
$$\|F_\infty\|=\sup_{\|x\|_1= 1}\|T_\infty x \|_{2(m-1)}\leq\frac{\pi}{\sqrt6}.$$
 This completes the proof.
  \end{proof}
 It follows from the definition (\ref{eq:13}) that $\mathcal{H}_\infty x^{m-1}$ is continuous, positively $(m-1)$-homogeneous, and so from the proof of Theorem \ref{th:21}, it also is bounded.
\begin{thm}\label{th:22} Let $\mathcal{H}_\infty$ be an $m$-order infinite dimensional Hilbert tensor and $f(x)=\mathcal{H}_\infty x^{m-1}$. Then
 $f$ is a bounded, continuous and positively $(m-1)$-homogeneous operator from $l^1$ into  $l^p$ ($1<p<\infty$).
 \end{thm}

{\bf Remark 1.} It is well known that Hilbert matrix $H_\infty$ is a bounded linear operator from $l^2$ into $l^2$ and $$\|H_\infty\|_2=\sup_{\|x\|_2= 1}\|H_\infty x \|_2=\pi.$$
For more details, see \cite{C1983}. Then if the Hilbert matrix $H_\infty$ is regarded as a bounded linear operator from $l^1$ into  $l^2$,  whether $\|H_\infty\|=\sup\limits_{\|x\|_1= 1}\|H_\infty x \|_2$ is exactly equal to $\frac{\pi}{\sqrt6}$ or another number? Furthermore, may the values of $\|T_\infty\|=\sup\limits_{\|x\|_1= 1}\|T_\infty x \|_2$ and $\|F_\infty\|=\sup\limits_{\|x\|_1= 1}\|F_\infty x \|_{2(m-1)}$ be worked out exactly?

\section{\bf Finite dimensional Hilbert tensors}

For $x\in \mathbb{R}^n$ and $\infty>p\geq1$, it is known well that $$\|x\|_p=\left(\sum_{i=1}^n|x_i|^p\right)^{\frac1p}$$ is the norm defined on $\mathbb{R}^n$ for each $p\geq1$ and  \begin{equation}\label{eq:31}\|x\|_p\leq\|x\|_r\leq n^{\frac1r-\frac1p}\|x\|_p\mbox{ for }p>r.\end{equation}  Then for a continuous, positively homogeneous $T:\mathbb{R}^n\to \mathbb{R}^n$,  it is obvious that $$\|T\|_p=\max_{\|x\|_p=1}\|Tx\|_p$$ is the operator norm of $T$ for each $p\geq1$ (Song and Qi \cite{SQ13}).
When $(\mathcal{H}_n x^{m-1})^{[\frac1{m-1}]}$ is well defined for all $x\in \mathbb{R}^n$,  let $$F_n x=(\mathcal{H}_n x^{m-1})^{[\frac1{m-1}]}$$ and \begin{equation}\label{eq:32}T_n x= \begin{cases}\|x\|_2^{2-m}\mathcal{H}_n x^{m-1},\ x\neq\theta\\
\theta,\ \ \ \ \ \ \ \ \ \ \  \ \ \ \ \ \ \ \ \  \ x=\theta,\end{cases}\end{equation} where $x^{[\frac1{m-1}]}=(x_1^{\frac1{m-1}},x_2^{\frac1{m-1}},\cdots , x_n^{\frac1{m-1}})^T$ and $\theta=(0,0,\cdots,0)^T$. Clearly, both $F_n$ and $T_n$ are continuous and positively homogeneous.
The following Hilbert inequality is well known (Frazer \cite{F1946}).
 \begin{lem} \label{le:31} Let $x=(x_1,x_2,\cdots,x_n)^T\in \mathbb{R}^n$. Then \begin{equation}\label{eq:33}\sum_{i=1}^n\sum_{j=1}^n\frac{|x_i||x_j|}{i+j-1}\leq \left(n\sin\frac{\pi}{n}\right)\sum_{k=1}^nx_k^2=\|x\|_2^2n\sin\frac{\pi}{n}.\end{equation}
 \end{lem}
Recall that $\lambda\in \mathbb{C}$ is called an {\em eigenvalue of $\mathcal{H}_n$}, if there is a vector $x\in \mathbb{R}^n\setminus\{\theta\}$ such that
\begin{equation}\label{eq:34}\mathcal{H}_nx^{m-1}=\lambda x^{[m-1]}, \end{equation}
where $x^{[m-1]}=(x_1^{m-1},\cdots , x_n^{m-1})^T$, and call $x$ an {\em eigenvector} associated with  $\lambda$. We call such an eigenvalue {\em H-eigenvalue} if it is real and has a real eigenvector $x$, and call such a real eigenvector $x$ an {\em H-eigenvector}.
A number $\mu\in \mathbb{C}$ is called {\em $E$-eigenvalue of $\mathcal{H}_n$}, if there is a  vector $x\in \mathbb{R}^n\setminus\{\theta\}$ such that
 \begin{equation}\label{eq:35}\begin{cases}\mathcal{H}_nx^{m-1}=\mu x \\
 x^Tx=1,\end{cases} \end{equation}
 and call a  vector $x$ an {\em $E$-eigenvector} associated with $\mu$. If $x$ is real, then $\mu$ is also real. In this case, $\mu$ and $x$ are called {\em $Z$-eigenvalue} of  $\mathcal{H}_n$ and
  {\em $Z$-eigenvector}  associated with $\mu$, respectively. These concepts were first introduced by Qi \cite{LQ1} for  the
higher order symmetric tensors. Lim \cite{LL} independently
introduced the notion of eigenvalue for higher order tensors but restricted $x$ to be a real vector and $\lambda$ to be a real number.

Since Hilbert tensor is positive (all entries are positive) and symmetric, then the following conclusions (i) were easily obtained from Chang, Pearson and Zhang \cite{CPT1}, Qi \cite{LQ13}, Song and Qi \cite{SQ14} and Yang and Yang \cite{YY10,YY11}, and the conclusions (ii) can be obtained from Chang, Pearson and Zhang \cite{CPZ13} and Song and Qi \cite{SQ13}.
\begin{lem}\label{le:32} Let $\rho(F_n)$ and $\rho(T_n)$ respectively denote the largest modulus
of the eigenvalues of operators $F_n$ and $T_n$.   Then
\begin{itemize}
\item[(i)] $(\rho(F_n))^{m-1}$ is
 a positive H-eigenvalue of $\mathcal{H}_n$  with a positive $H-$eigenvector, i.e. all components are positive and \begin{equation}\label{eq:36} \rho(F_n)^{m-1}=\max\{\mathcal{H}_nx^m; x\in\mathbb{R}^n_+,\|x\|_m=1\};\end{equation}
\item[(ii)] $\rho(T_n)$ is
a positive $Z-$eigenvalue of $\mathcal{H}_n$  with a nonnegative $Z-$eigenvector and \begin{equation}\label{eq:37} \rho(T_n)=\max\{\mathcal{H}_nx^m; x\in\mathbb{R}^n,\|x\|_2=1\}.\end{equation}
\end{itemize}
\end{lem}
Now we give the upper bounded of the eigenvalues of operators $F_n$ and $T_n$.
\begin{thm} \label{th:33} Let $\mathcal{H}_n$ be an $m$-order $n$-dimensional Hilbert tensor. Then
\begin{itemize}
\item[(i)]  for all $E$-eigenvalues (Z-eigenvalues) $\mu$ of Hilbert tensor $\mathcal{H}_n$, $$|\mu|\leq \rho(T_n)\leq n^{\frac{m}2}\sin\frac{\pi}{n};$$
\item[(ii)] for all eigenvalues (H-eigenvalues) $\lambda$ of Hilbert tensor $\mathcal{H}_n$,  $$|\lambda|\leq \rho(F_n)^{m-1}\leq n^{m-1}\sin\frac{\pi}{n}.$$
\end{itemize}
\end{thm}
 \begin{proof} For $x\in\mathbb{R}^n\setminus\{\theta\}$, it follws from Lemma \ref{le:31} that
 \begin{align}|\mathcal{H}_n x^m|=&\left|\sum_{i_1,i_2,\cdots,i_m=1}^n\frac{x_{i_1}x_{i_2}\cdots x_{i_m}}{i_1+i_2+\cdots+i_m-m+1}\right|\nonumber\\
 \leq& \sum_{i_1,i_2,\cdots,i_m=1}^n\frac{|x_{i_1}x_{i_2}\cdots x_{i_m}|}{i_1+i_2+\underbrace{1+\cdots+1}_{m-2}-m+1}\nonumber\\
 =&\sum_{i_1,i_2,\cdots,i_m=1}^n\frac{|x_{i_1}||x_{i_2}|\cdots |x_{i_m}|}{i_1+i_2-1}\nonumber\\
 =&\left(\sum_{i_1=1}^n\sum_{i_2=1}^n\frac{|x_{i_1}||x_{i_2}|}{i_1+i_2-1}\right)\sum_{i_3,i_4,\cdots,i_m=1}^n|x_{i_3}||x_{i_4}|\cdots |x_{i_m}|\nonumber\\
\leq& \left(\|x\|_2^2n\sin\frac{\pi}{n}\right)\left(\sum_{i=1}^n|x_{i}|\right)^{m-2}\nonumber\\
 =&\left(n\sin\frac{\pi}{n}\right)\|x\|_2^2\|x\|_1^{m-2}.\label{eq:38}
\end{align}

 (i) From (\ref{eq:31}), it follows that $\|x\|_1\leq \sqrt{n}\|x\|_2$ for $x\in \mathbb{R}^n$. Then  we have
 $$\begin{aligned}\mathcal{H}_n x^m\leq& (n\sin\frac{\pi}{n})\|x\|_2^2\|x\|_1^{m-2}\\ \leq& (n\sin\frac{\pi}{n})n^{\frac{m-2}2}\|x\|_2^m\\
 =&(n^{\frac{m}2}\sin\frac{\pi}{n})\|x\|_2^m,
 \end{aligned}$$
 and hence, for $x\in\mathbb{R}^n\setminus\{\theta\}$, $$\mathcal{H}_n\left(\frac{x}{\|x\|_2}\right)^m=\frac1{\|x\|_2^m}\mathcal{H}_n x^m\leq n^{\frac{m}2}\sin\frac{\pi}{n}.$$
 It follows from (\ref{eq:37}) of Lemma \ref{le:32} (ii) that
 $$\rho(T_n)\leq n^{\frac{m}2}\sin\frac{\pi}{n}.$$

 (ii) From (\ref{eq:31}), it follows that $$\|x\|_2\leq n^{\frac12-\frac1m}\|x\|_m\mbox{ and }\|x\|_1\leq n^{1-\frac1m}\|x\|_m$$ for $m\geq2$ and $x\in \mathbb{R}^n$. Then by (\ref{eq:38}), we have
  $$\begin{aligned}\mathcal{H}_n x^m\leq& \left(n\sin\frac{\pi}{n}\right)\|x\|_2^2\|x\|_1^{m-2}\\ \leq& \left(n\sin\frac{\pi}{n}\right)\left(n^{1-\frac2m}\|x\|_m^2\right)\left(n^{(1-\frac1m)(m-2)}\|x\|_2^{m-2}\right)\\
  =&\left(n\sin\frac{\pi}{n}\right)\left(n^{m-2}\|x\|_m^m\right)\\
  =&\left(n^{m-1}\sin\frac{\pi}{n}\right)\|x\|_m^m,
  \end{aligned}$$
  and hence, for $x\in\mathbb{R}^n\setminus\{\theta\}$, $$\mathcal{H}_n\left(\frac{x}{\|x\|_m}\right)^m=\frac1{\|x\|_m^m}\mathcal{H}_n x^m\leq n^{m-1}\sin\frac{\pi}{n}.$$
  It follows from (\ref{eq:36}) of Lemma \ref{le:32} (i) that
  $$\rho(F_n)^{m-1}\leq n^{m-1}\sin\frac{\pi}{n}.$$
 This completes the proof.
  \end{proof}

 Recall that
 A $m$-order $r$-dimensional $\mathcal{B}$ is called {\em a principal sub-tensor}  of a $m$-order $n$-dimensional tensor $\mathcal{A}=(\mathcal{A}_{i_1\cdots i_m})$ ($r\leq n$), if $\mathcal{B}$ consists of $r^m$ elements in $\mathcal{A} = (a_{i_1\cdots i_m})$: for a set $\mathcal{N}$ that composed of $r$ elements in $ \{1,2,\cdots , n\}$,
 $$\mathcal{B} = (\mathcal{A}_{i_1\cdots i_m}),\mbox{ for all } i_1, i_2, \cdots, i_m\in \mathcal{N}.$$ The concept were first introduced and used by Qi \cite{LQ1} for  the higher order symmetric tensor. Clearly, a $m$-order $n_1$-dimensional Hilbert tensor $\mathcal{H}_{n_1}$ is a principal sub-tensor  of $m$-order $n_2$-dimensional Hilbert tensor $\mathcal{H}_{n_2}$ if $n_1\leq n_2$.
\begin{thm} \label{th:34}   If $n< k$, then
  $$ \rho(F_{n})< \rho(F_{k})\mbox{ and } \rho(T_{n})\leq \rho(T_{k}).$$
\end{thm}
\begin{proof}  Since $n<k$, then  $\mathcal{H}_{n}$ is a principal sub-tensor  of $\mathcal{H}_{k}$. It follows from Lemma \ref{le:32} (i) that $\rho(F_{n})^{m-1}$ is a positive eigenvalue of $\mathcal{H}_{n}$ with positive eigenvector $x_{(n)}=(x^{(n)}_1,\cdots,x^{(n)}_n)$,  and hence $\rho(F_{n})^{m-1}$ is an eigenvalue of of $\mathcal{H}_{k}$ with corresponding eigenvector $x'=(x^{(n)}_1,\cdots,x^{(n)}_n,\underbrace{0,\cdots,0}_{k-n})$. Since $\rho(F_{k})^{m-1}$ is  positive eigenvalue  of $\mathcal{H}_{k}$ with positive eigenvector $x_{(k)}=(x^{(k)}_1,\cdots,x^{(k)}_k)$ by Lemma \ref{le:32} (i), then $$\rho(F_{n})^{m-1}< \rho(F_{k})^{m-1},$$
 and hence, $\rho(F_{n})< \rho(F_{k}).$

Similarly, applying Lemma \ref{le:32} (ii), we also have $$\rho(T_{n})\leq \rho(T_{k}).$$ The desired conclusion follows.
\end{proof}

{\bf Remark 2.} (i) In Theorem \ref{th:34}, the monotonicity of the spectral radius with respect to the dimensionality $n$ is proved. Then  whether or not  the eigenvector $x_{(n)}$ associated with the spectral radius is the same monotonicity.

(ii) In Theorem \ref{th:33}, the upper bounds of two classes of spectral radii are established.  It is not clear whether these two upper bounds may be  attained or only one of these two upper bounds may be attained or both cannot be attained.

\section*{\bf Acknowledgment}
The authors would like to thank Dr. Guoyin Li, Prof. Yimin Wei and
Mr. Weiyang Ding for their valuable suggestions.   In particular, in
the original draft, we only proved that $\mathcal{H}_n$ and
$\mathcal{H}_\infty$ are positive semi-definite.   Dr. Guoyin Li
suggested the current proof for showing that $\mathcal{H}_n$ and
$\mathcal{H}_\infty$ are positive definite.


\begin{thebibliography}{99}
\bibitem{CPZ13}K.C. Chang, K. Pearson, and T. Zhang, \textit{Some variational principles for $Z$-eigenvalues of nonnegative tensors}, Linear Algebra Appl. 438(2013) 4166-4182.
\bibitem{CPT1}K.C. Chang, K. Pearson, and T. Zhang, \textit{Perron-Frobenius theorem for nonnegative tensors}, Commun. Math. Sci. 6(2008)  507-520.
\bibitem{C1983}Man-Duen Choi, \textit{Tricks for Treats with the Hilbert Matrix}, Amer. Math. Monthly 90(1983),301-312.
\bibitem{F1946}H. Frazer,  \textit{Note on Hilbert's Inequality}, J. London Math. Soc. (1946) s1-21 (1): 7-9.
\bibitem{FNSS}S. Fucik,  J. Necas, J. Soucek and V. Soucek,  \textit{Spectral Analysis of Nonlinear Operators}, Lecture Notes in Mathematics 346, Springer-Verlag, Berlin, Heidelberg, New York, 1973.
\bibitem{H1894} D. Hilbert, \textit{Ein Beitrag zur Theorie des Legendre'schen Polynoms}, Acta Mathematica (Springer Netherlands) 18(1894) 155-159.
\bibitem{I1936}A. E. Ingham,  \textit{A Note on Hilbert's Inequality}, J. London Math. Soc. (1936) s1-11 (3): 237-240.
\bibitem{K1957}T. Kato, \textit{On the Hilbert Matrix},
Proc. American Math. Soc. 8(1)(1957) 73-81.
\bibitem{LL}L.H. Lim, \textit{Singular values and eigenvalues of tensors: A variational approach}, in: Proc. 1st IEEE International workshop on computational advances of multi-tensor adaptive processing, Dec. 13-15, 2005, pp. 129-132.
\bibitem{M1950}W. Magnus \textit{On the Spectrum of Hilbert's Matrix}, American Journal of Mathematics,  72(4)(1950)  699-704.

\bibitem{LQ1}L. Qi, \textit{Eigenvalues of a real supersymmetric tensor}, J. Symbolic Comput. 40(2005) 1302-1324.
\bibitem{LQ13}L. Qi, \textit{Symmetric nonnegative tensors and copositive tensors}, Linear Algebra Appl., 439(2013) 228-238.
\bibitem{LQH}L. Qi,  \textit{Hankel Tensors: Associated Hankel Matrices and Vandermonde Decomposition}, (2013) arXiv:1310.5470.
 \bibitem{SQ13}Y. Song and L. Qi, \textit{Positive eigenvalue-eigenvector of nonlinear positive mappings}, Frontiers of Mathematics in China,  9(1)(2014) 181-199
\bibitem{SQ14}Y. Song and L. Qi, \textit{Spectral properties of positively homogeneous operators induced by higher order tensors},  SIAM. J. Matrix Anal. \& Appl., 34(4)(2013) 1581-1595.
\bibitem{T1949}O. Taussky, \textit{A remark concerning the characteristic roots of the finite segments of the Hilbert matrix,} Quarterly Journal of Mathematies, Oxford ser., vol. 20 (1949)  80-83.
\bibitem{YY11}Q. Yang and Y. Yang, \textit{Further Results for Perron-Frobenius Theorem for Nonnegative Tensors II}, SIAM. J. Matrix Anal. \& Appl. 32(4)(2011) 1236-1250.
\bibitem{YY10}Y. Yang and Q. Yang, \textit{Further Results for Perron-Frobenius Theorem for Nonnegative Tensors}, SIAM. J. Matrix Anal. \& Appl., 31(5)(2010) 2517-2530.
\end{thebibliography}
\end{document}